\documentclass[review]{elsarticle}
\journal{}
\biboptions{authoryear}

\usepackage[nolists,tablesfirst,fighead,tabhead]{endfloat} 

\usepackage{hyperref}
\hypersetup{
  colorlinks   = true, 
  urlcolor     = blue, 
  linkcolor    = blue, 
  citecolor   = blue!70 
}

\usepackage[margin=2.5cm]{geometry}
\usepackage{soul}
\usepackage[numbered]{bookmark}
\usepackage{amsmath,amssymb,amsthm}
\usepackage{soul}
\usepackage[ruled,vlined]{algorithm2e}

\usepackage{float} 
\floatstyle{plaintop} 

\usepackage{booktabs}

\usepackage{xfrac}
\usepackage{color, colortbl}
\usepackage[dvipsnames]{xcolor}
\usepackage[makeroom]{cancel}
\usepackage{bbding}
\usepackage{rotating}
\usepackage{multirow}
\usepackage{array}
\usepackage{graphicx}
\usepackage{animate}
\usepackage{tikz}
\usepackage{pgfplots}
\usepackage{mathtools}
\usetikzlibrary{fit,shapes.misc}

\renewcommand\appendix{\par
  \setcounter{section}{0}
  \setcounter{subsection}{0}
  \setcounter{figure}{0}
  \setcounter{table}{0}
  \renewcommand\thesection{Appendix \Alph{section}}
  \renewcommand\thefigure{\Alph{section}\arabic{figure}}
  \renewcommand\thetable{\Alph{section}\arabic{table}}
}

\theoremstyle{plain}\newtheorem{theorem}{Theorem}[section]
\theoremstyle{plain}
\theoremstyle{plain}
\theoremstyle{plain}\newtheorem{lemma}[theorem]{Lemma}
\theoremstyle{remark}
\theoremstyle{definition}

\pgfplotsset{compat=1.14}

\usetikzlibrary{patterns}

\begin{document}
\begin{frontmatter}

\title{Comments on the "Optimal strategy of deteriorating items with capacity constraints under two-levels of trade credit policy"}

\author[1st_address]{Sunil Tiwari\href{https://orcid.org/0000-0002-0499-2794}{\includegraphics[scale=.6]{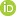}}} \corref{mycorrespondingauthor}
\cortext[mycorrespondingauthor]{Corresponding author}
\ead{sunil.tiwari047@gmail.com}

\author[2nd_address]{Masih Fadaki} 
\ead{masih.fadaki@rmit.edu.au}
\author[3rd_address]{Anuj Kumar Sharma} 
\ead{anujsharma1920@gmail.com}

\address[1st_address]{\textcolor{white}{"}Department of Industrial Systems Engineering and Management, National University of Singapore, 1 Engineering Drive 2, 117576, Singapore\textcolor{white}{"}}
\address[2nd_address]{\textcolor{white}{"}School of Business IT \& Logistics, RMIT University, Melbourne, VIC 3000, Australia\textcolor{white}{"}}
\address[3rd_address]{Department of Mathematics, Shyam Lal College, University of Delhi, India}

\begin{abstract}
This technical note rectified the mathematical and conceptual errors present in \citet{liao2014}. \citet{liao2014} proposed an EOQ model under two-levels trade credit policy considering limited storage capacity whereby the supplier provides a permissible delay period ($M$) to the retailer, and the retailer also offers a permissible delay period ($N$) (where $M>N$) to its customers. In the current technical note, we point out some defects of their model from the logical viewpoints of mathematics regarding both interest charged and interest earned. Furthermore, as an example, one of the affected numerical results is re-evaluated.\\
\end{abstract}

\begin{keyword}
Trade credit financing \sep Permissible delay in payments \sep Deterioration \sep Limited storage capacity \sep Inventory.
\end{keyword}

\end{frontmatter}


\section{Introduction}\label{sec:Intr}
To reflect the real business environment and addressing the inter-dependencies among the operations and financing processes, there is recently a trend of optimising inventory while the trade credit is also taken into consideration. \citet{liao2014} presented an inventory model for the perishable items with two-levels of trade credit and limited storage capacity. In this two-echelon supply chain inventory model, a supplier provides the products to a retailer and the retailer sells them to its customers. Under two levels of trade credit, supplier offers a fixed credit period to the retailer to settle the account and the retailer provides some trade credit to his customers. Beyond the allowable credit period, the retailer is charged as per agreed interest rate while she receives the interest regarding the customer's \textcolor{white}{"}outstanding amount if it is not paid within the allowable delay period.\textcolor{white}{"}

While the study contributed an interesting idea of considering the capacity constrained in the context of inventory optimisation under trade credit policy, we have identified a major issue in developing the mathematical expression regarding the interest earned of the second case ($N < T \le M $). Furthermore, a number of minor issues including an error in the Notations section and an incorrect mathematical expression for the annual deteriorating cost is spotted. Since the outlined major issue has a significant impact on the numerical results and the corresponding recommendations, we also amend the impacted computational results.

\section{Notations}

We have reviewed the notations of the original paper and compared the provided notations with the terms used in the paper. A minor issue exists in the notations with respect to the formula of $W^*$. The unit selling price of the item is denoted by '$s$' in the notation list while the $W^*$ is formulated as $pDM+\frac{pI_eDM^2}{2}$ in which '$p$' is used for the unit selling price. This expression should be updated to $sDM+\frac{sI_eDM^2}{2}$.

\section{Model}

Suppose there is a delay in payment time, denoted as \textit{M}. The inventory level $I(t)$ at a time $t\in \left[0,{\rm \; }T \right]$ can be formulated as:

\begin{flalign}
\textcolor{white}{"}
\frac{dI\left(t\right)}{dt} +\theta I\left(t\right)=-D;{\rm \; \; \; \; }0\le t\le T
\textcolor{white}{"}
\end{flalign}

\vspace{5mm}

\textcolor{white}{"}
Given the boundary condition as $I(T)=0$, we get
\textcolor{white}{"}

\begin{flalign}
I\left(t\right)=\frac{D}{\theta } \left(e^{\theta \left(T -t\right)} -1\right);{\rm \; \; \; \; }0\le t\le T  
\end{flalign}

\vspace{5mm}

At time $t = 0$, \textit{I}(0) = \textit{Q} and thus:
\begin{align} \label{eq:Qj}
Q &= \frac{D}{\theta } \left(e^{\theta T } -1\right) \\
\Rightarrow \; \frac{dQ }{dT } &= De^{\theta T }  \notag
\end{align}

\vspace{5mm}

Considering Equation \ref{eq:Qj}, the time at which \textit{W} inventory is exhausted (\textit{T${}_{a}$}) can be computed as:

\begin{align} \label{eq:Ta}
W &=\frac{D}{\theta } \left(e^{\theta T_{a} } -1\right) \\
\Rightarrow T_{a} & =\frac{1}{\theta } \ln \left(\frac{\theta W}{D} +1\right) \nonumber
\end{align}

\vspace{5mm}

\subsection{Estimation of total annual cost components}

There is another minor issue regarding the mathematical expression of annual deteriorating cost as it has been developed as $\frac{Dh(e^{\theta T}-\theta T -1)}{\theta T}$ in \citeauthor{liao2014}'s \citeyearpar{liao2014} model; however, the correct term is:

\begin{align*}
D_{T } &=\frac{c\left(Q -DT \right)}{T } \\
&=\frac{cD}{\theta T } \left\{e^{\theta T } -\theta T -1\right\}
\end{align*}

\vspace{5mm}

The interest earned per year, \textbf{Case 2: $N < T \le M $}

The incomplete term in \citeauthor{liao2014}'s \citeyearpar{liao2014} model pertains to computing the annual interest earned for the second case where "$N < T \le M$". This is the case where \citet{liao2014} did not consider the interest earned during \textit{N} to \textit{T}. As depicted in Figure \ref{fig:case2}, annual interest earned should be calculated for both periods of $N$ to $T$ and $T$ to $M$. They formulated the interest earned for this period as $\frac{sI_{e}D (2MT-N^2-T^2)}{2T }$; however, the correct interest earned for this case should be formulated as:

\textcolor{white}{"}
\begin{align*}
&=\frac{sI_{e} \left(DN+DT \right)\left(T -N\right)}{2T } +\frac{1}{T } \left\{sDT+\frac{sDI_{e} \left(T^{2} -N^{2} \right)}{2} \right\}\left(M -T \right)I_{e} \\
&=\frac{sDI_{e} \left(T^{2} -N^{2} \right)+\left(2sDT+sD\left(T^{2} -N^{2} \right)I_{e} \right)I_{e} \left(M -T \right)}{2T } 
\end{align*}
\textcolor{white}{"}

\begin{figure}[H]
\centering
\begin{tikzpicture}[scale=1.3]
    \draw [thick, ->] (0,0) -- (7,0) node[anchor=north west] {time};
    \draw [thick, ->] (0,0) -- (0,5);
    
    \draw [thick, -] (0,0) -- (4,2.5);
    \draw [thick, -] (4,2.5) -- (6,2.5);
    \draw [thick, -] (4,3.5) -- (6,3.5);
    
    \draw [thick, -] (2.5,1.5625) -- (2.5,0);
    \draw [thick, -] (4,2.5) -- (4,0);
    \draw [thick, -] (6,2.5) -- (6,0);
    \draw [thick, -] (4,2.5) -- (4,3.5);
    \draw [thick, -] (6,2.5) -- (6,3.5);
    
    \draw [dashed, -] (0,1.5625) -- (2.5,1.5625);
    \draw [dashed, -] (0,2.5) -- (4,2.5);
    
    \coordinate [label=$0$] (L1) at (0,-5mm);
    \coordinate [label=$N$] (L2) at (2.5,-5mm);
    \coordinate [label=$T$] (L3) at (4,-5mm);
    \coordinate [label=$M$] (L4) at (6,-5mm);
    
    \coordinate [label=$DN$] (L5) at (-0.3,1.4);
    \coordinate [label=$DT$] (L5) at (-0.3,2.4);
    
     
     \fill[pattern=horizontal lines, pattern color=blue] (4,0) rectangle (6,2.5);
     
     \fill[pattern=horizontal lines, pattern color=blue] (2.5,0) -- (2.5,1.5625) -- (4,2.5) -- (4,0) -- cycle;
     
     \fill[pattern=vertical lines, pattern color=green] (4,2.5) rectangle (6,3.5);
     
    \draw [color= gray] (5,4) rectangle (7,5);
    
    \draw[pattern= horizontal lines, pattern color=blue] (5.1,4.1) rectangle (5.5,4.4) node at (6.2,4.25) {\tiny{Interest earned}};
    
    \draw[pattern= vertical lines, pattern color=green] (5.1,4.6) rectangle (5.5,4.9) node at (6.2,4.75) {\tiny{Interest charged}};
    
\end{tikzpicture}
\caption{Demonstration of annual interest earned for case 2.} \label{fig:case2}
\end{figure}
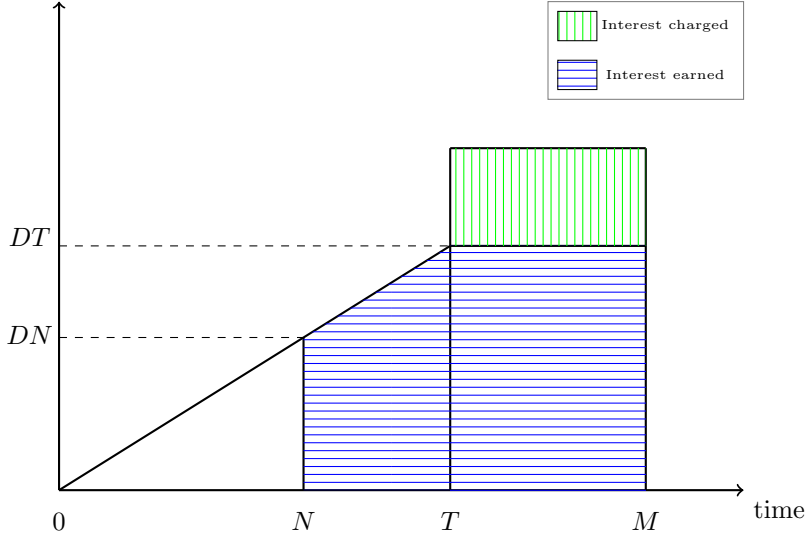

\subsection{Implications of corrected term on the total annual cost}

Putting together all components of retailer's total annual cost, is can be expressed as:

\begin{quote}
\begin{tabular}{ll}
     \textit{TC}(\textit{T}) = &  "ordering cost + deterioration cost + stock-holding cost in RW \\
     & + stock-holding cost in OW + interest paid -- interest earned"
\end{tabular}
\end{quote}

Given the possible values of \textit{N}, \textit{M}, and \textit{T}, the retailer's total annual cost is estimated for the following configurations:

\vspace{5mm}

\textbf{Configuration 1: \textcolor{white}{"}$T_{a} <N < M  < W^*$\textcolor{white}{"}}

Regarding this configuration, the total annual cost can be estimated as follows:

\[TC\left(T \right)=\left\{\begin{array}{l} {TC_{1} \left(T \right);{\rm \; \; \; \; }if{\rm \; }0<T \le T_{a} } \\ {TC_{2} \left(T \right);{\rm \; \; \; \; }if{\rm \; }T_{a} <T \le N} \\ {TC_{3} \left(T \right);{\rm \; \; \; \; }if{\rm \; }N<T \le M } \\ {TC_{4} \left(T \right);{\rm \; \; \; \; }if{\rm \; }M<T \le W^*  } \\ {TC_{5} \left(T \right);{\rm \; \; \; \; }if{\rm \; }W^* <T } \end{array}\right. \] 

\vspace{5pt}

In this configuration, the total cost regarding the third time window ($N < T \le M$) should be amended to:

\begin{flalign} 
\textcolor{white}{"}
TC_{3} \left(T \right) &= \frac{A}{T } +\frac{D\left(k+c\theta \right)}{\theta ^{2} T } \left(e^{\theta T } -\theta T -1\right)-\frac{\left(k-h\right)}{\theta ^{2} T } \left\{D\left(e^{\theta T_{a} } -\theta T_{a} -1\right)+\theta ^{2} W\left(T -T_{a} \right)\right\}  \\ 
&\hphantom{{}=} -\frac{1}{2T } \left\{sDI_{e} \left(T^{2} -N^{2} \right)+\left(2sDT +sD\left(T^{2} -N^{2} \right)I_{e} \right)I_{e} \left(M -T \right)\right\} & \nonumber
\textcolor{white}{"}
\end{flalign}

\vspace{7pt}

\textbf{Configuration 2: \textcolor{white}{"}$N<T_{a} <M < W^*$\textcolor{white}{"}}

With respect to this configuration, the total annual cost can be estimated as:

\textcolor{white}{"}
\[TC\left(T \right)=\left\{\begin{array}{l} {TC_{1} \left(T \right);{\rm \; \; \; \; }if{\rm \; }0<T \le N} \\ {TC_{6} \left(T \right);{\rm \; \; \; \; }if{\rm \; }N<T \le T_{a} } \\ {TC_{3} \left(T \right);{\rm \; \; \; \; }if{\rm \; }T_{a} <T \le M } \\ {TC_{4} \left(T \right);{\rm \; \; \; \; }if{\rm \; }M <T {\rm \; } \le W^* } \\ {TC_{5} \left(T \right);{\rm \; \; \; \; }if{\rm \; }W^* <T {\rm \; }} \end{array}\right. \] 
\textcolor{white}{"}
where

\begin{flalign}
\textcolor{white}{"}
TC_{3} \left(T \right) &= \frac{A}{T } +\frac{D\left(k+c\theta \right)}{\theta ^{2} T } \left(e^{\theta T } -\theta T -1\right)-\frac{\left(k-h\right)}{\theta ^{2} T } \left\{D\left(e^{\theta T_{a} } -\theta T_{a} -1\right)+\theta ^{2} W\left(T -T_{a} \right)\right\}  \\ 
&\hphantom{{}=} -\frac{1}{2T } \left\{sDI_{e} \left(T^{2} -N^{2} \right)+\left(2sDT +sD\left(T^{2} -N^{2} \right)I_{e} \right)I_{e} \left(M -T \right)\right\} & \nonumber
\textcolor{white}{"}
\end{flalign}

\vspace{7pt}

\subsection{Probing the convexity for total annual cost function}

In this section, we investigate the convexity of the total annual cost functions ($TC_3(T)$).

\begin{theorem} \label{thrm:T1}
$TC_{3} \left(T \right)$ is convex on $T >0$.
\end{theorem}

\begin{proof}
Considering the first and the second derivatives of the revised $TC_3(T)$:

\begin{flalign*}
\textcolor{white}{"}
TC_{3} \left(T \right) &=\frac{A}{T } +\frac{D\left(k+c\theta \right)}{\theta ^{2} T } \left(e^{\theta T } -\theta T -1\right)-\frac{\left(k-h\right)}{\theta ^{2} T } \left\{D\left(e^{\theta T_{a} } -\theta T_{a} -1\right)+\theta ^{2} W\left(T -T_{a} \right)\right\} & \\
&\hphantom{{}=}-\frac{1}{2T } \left\{DsI_{e} \left(T^{2} -N^{2} \right)+\left(2DT s+Ds\left(T^{2} -N^{2} \right)I_{e} \right)I_{e} \left(M -T \right)\right\} &
\end{flalign*}
\textcolor{white}{"}

\textcolor{white}{"}
\begin{flalign*}
TC_{3}^{'} \left(T \right) &=\frac{1}{\left(\theta T \right)^{2} } \left[-\theta ^{2} A+D\left(k+c\theta \right)\left(\theta T e^{\theta T } -e^{\theta T } +1\right)-D\left(k-h\right)\left(\theta T_{a} e^{\theta T_{a} } -e^{\theta T_{a} } +1\right)\right] & \\ 
&\hphantom{{}=}-\frac{DsI_{e} }{2T^{2} } \left\{\left(N^{2} -T^{2} \right)+\left(1+T I_{e} \right)+\left(N^{2} +T^{2} \right)I_{e} \right\} &
\end{flalign*}
\textcolor{white}{"}

\textcolor{white}{"}
\begin{flalign*}
TC_{3}^{''} \left(T \right) &=\frac{2A}{T^{3} } +\frac{D\left(k+c\theta \right)}{\theta ^{2} T^{3} } \left\{\theta ^{2} T^{2} e^{\theta T } -2\theta T e^{\theta T } +2e^{\theta T } -2\right\}+\frac{2D\left(k-h\right)}{\theta ^{2} T^{3} } \left(\theta T_{a} e^{\theta T_{a} } -e^{\theta T_{a} } +1\right) & \\ 
&\hphantom{{}=}+\frac{DsI_{e} }{2T^{3} } \left\{T^{2} +N^{2} +I_{e} \left(N^{2} -T^{2} +2T^{3} \right)\right\} &
\end{flalign*}
\textcolor{white}{"}

\textcolor{white}{"}
\begin{flalign*}
TC_{3}^{'} \left(T \right) &= TC_{2}^{'} \left(T \right)-\frac{DsI_{e} }{2T^{2} } \left\{\left(N^{2} -T^{2} \right)+\left(1+T I_{e} \right)+\left(N^{2} +T^{2} \right)I_{e} \right\} &
\end{flalign*}
\textcolor{white}{"}

\textcolor{white}{"}
\begin{flalign} \label{eq:thrm3.5}
TC_{3}^{''} \left(T \right) &=TC_{2}^{''} \left(T \right)+\frac{DsI_{e} }{2T^{3} } \left\{T^{2} +N^{2} +I_{e} \left(N^{2} -T^{2} +2T^{3} \right)\right\} & 
\end{flalign}
\textcolor{white}{"}

\begin{lemma} \label{lemma:forT3}
$T^{2} +N^{2} +I_{e} \left(N^{2} -T^{2} +2T^{3} \right)$$\mathrm{>}$0
\end{lemma}

\begin{proof}
Let 
\begin{equation} \label{eq:lemma3.6}
f(T )=T^{2} +N^{2} +I_{e} \left(N^{2} -T^{2} +2T^{3} \right)
\end{equation}

\[f^{'} (T )=2T -2I_{e} T +6T^{2} \] 
\[f^{'} (T )=2T (1-I_{e} )+6T^{2} \] 
\[f^{'} (T )>0 \, \, \therefore \, \, 0<I_{e} <1\]

This implies that $f(T )$ is increasing function of $T $ and $f(0)=N^{2} (1+I_{e} )>0$.

Hence, $f(T )>0$, when $T >0$

This completes the proof of Lemma \ref{lemma:forT3}.
\end{proof}

\textit{Continuing the proof of Theorem \ref{thrm:T1}.}

From Equations \ref{eq:thrm3.5} and \ref{eq:lemma3.6}:
\begin{flalign*}
TC_{3}^{''} \left(T \right) &=TC_{2}^{''} \left(T \right)+\frac{DsI_{e} }{2T^{3} } f(T ) & 
\end{flalign*}

From Lemma \ref{lemma:forT3} and as $TC_{2} \left(T \right)$ is convex (See \citet{liao2014}'s paper), therefore $TC_{3}^{''} \left(T \right)>0$

Hence,  $TC_{3} \left(T \right)$ is convex on $T >0$.

This completes the proof.

\end{proof}


\section{Updated Decision rule}
From Theorem 3.1, we get,

\textcolor{white}{"}
\begin{flalign*}
TC_{3}^{'} \left(T \right) &=\frac{1}{\left(\theta T \right)^{2} } \left[-\theta ^{2} A+D\left(k+c\theta \right)\left(\theta T e^{\theta T } -e^{\theta T } +1\right)-D\left(k-h\right)\left(\theta T_{a} e^{\theta T_{a} } -e^{\theta T_{a} } +1\right)\right] & \\ 
&\hphantom{{}=}-\frac{DsI_{e} }{2T^{2} } \left\{\left(N^{2} -T^{2} \right)+\left(1+T I_{e} \right)+\left(N^{2} +T^{2} \right)I_{e} \right\} &
\end{flalign*}
\textcolor{white}{"}

After simplifying above equation, we get

\textcolor{white}{"}
\begin{flalign*}
TC_{3}^{'} (T ) &=\frac{1}{(\theta T )^{2} } [-\theta ^{2} A+D(k+c\theta )(\theta T e^{\theta T } -e^{\theta T } +1)-D(k-h)(\theta T_{a} e^{\theta T_{a} } -e^{\theta T_{a} } +1) & \\ 
&\hphantom{{}=}-\frac{DsI_{e}\theta^{2} }{2} \{(N^{2} -T^{2} )+(1+T I_{e} )+(N^{2} +T^{2} )I_{e} \} ] &
\end{flalign*}
\textcolor{white}{"}

Now, replacing $T=M$ in above equation, we get

\textcolor{white}{"}
\begin{flalign*}
TC_{3}^{'} (M ) &=\frac{1}{(\theta M )^{2} } [-\theta ^{2} A+D(k+c\theta )(\theta M e^{\theta M } -e^{\theta M } +1)-D(k-h)(\theta T_{a} e^{\theta T_{a} } -e^{\theta T_{a} } +1) & \\ 
&\hphantom{{}=}-\frac{DsI_{e}\theta^{2} }{2} \{(N^{2} -M^{2} )+(1+M I_{e} )+(N^{2} +M^{2} )I_{e} \}] &
\end{flalign*}
\textcolor{white}{"}

\begin{flalign*}
TC_{3}^{'} \left(M \right) &=\frac{\Delta_3}{\left(\theta M \right)^{2} } &
\end{flalign*}

where

\textcolor{white}{"}
\begin{flalign*}
\Delta_3 &=[-\theta ^{2} A+D(k+c\theta )(\theta M e^{\theta M } -e^{\theta M } +1)-D(k-h)(\theta T_{a} e^{\theta T_{a} } -e^{\theta T_{a} } +1) & \\ 
&\hphantom{{}=}-\frac{DsI_{e}\theta^{2} }{2} \{(N^{2} -M^{2} )+(1+M I_{e} )+(N^{2} +M^{2} )I_{e} \}] &
\end{flalign*}
\textcolor{white}{"}

\section{Implications on the Final Result}

In Tables 1 and 2 of Liao's study, where the optimal value of $T$ obtains from $\mathrm{TC_3}$, the value of optimal $T$, the corresponding decision rule, and value of total cost need to be updated. To illustrate this matter, we have re-calculated the values of optimal $T$ and the corresponding total cost for row 3 of Table 1 in Liao's study (Table \ref{tab:implications}).

\clearpage

\textbf{Parameters of experiment regarding row 3 of Table 1}

\vspace{5mm}

\begin{tabular}{lllp{0.8in}llll} 
$\theta$&=&0.00001&&	$h$&=&1& \\
$D$&=&9000000&&	$\mathrm{I_e}$&=&0.000005& \\
$A$&=&1224.04585&&	$\mathrm{I_p}$&=&0.15& \\
$c$&=&1.9999&&	$N$&=&0.0161& \\
$s$&=&2&&	$M$&=&0.0165& \\
$k$&=&1.1&&	$W$&=&144900& \\
\end{tabular}

\vspace{5mm}

\begin{table}[h]
\centering
\caption{Implications of updated $\mathrm{TC_3}$ on the final result.}
\label{tab:implications}
\scalebox{0.8}{
\begin{tabular}{@{}lcccccccccccl@{}}
\toprule
        & Table & Row No & Dec. Rule & $T_a$  & W*     & $\Delta_1$                    & $\Delta_2$                    & $\Delta_3$                       & $\Delta_4$                    & $\Delta_5$                       & T*                        & TVC(T*)    \\ \midrule
Liao    & 1     & 3      & (B1)      & 0.0161 & 0.0165 & \textless{}0 & \textless{}0 & \textgreater{}0 & \textless{}0 & \textless{}0    & $\mathrm{T^*_3}$=0.0161   & 148020     \\
Updated &       &        & (B2)      &        &        &                               &                               & $\sim$0                          &                               & \textgreater{}0 & $\mathrm{T^*_3}$=1.000129 & 4936419.16 \\ \bottomrule
\end{tabular}
}
\end{table}

\section{Conclusion}
In this technical note, we reviewed the \citeauthor{liao2014}'s \citeyearpar{liao2014} model in which an inventory model was developed by considering that the storage capacity is limited and there are two levels of trade credit. The unavoidable complexity emanating from the difference between permissible delay regarding the periods of interest earned and interest charged requires the decision maker to seek for an optimal inventory model. We identified some minor and one major issue regarding the formulation of $\mathrm{TC_3}$ as the original study did not consider the interest earned during \textit{N} to \textit{T}. The errors have been corrected and the impact of updated terms on the computational result has been presented.

\clearpage

\section*{\refname}
\footnotesize{
\bibliography{Optimal_Inventory}

\begin{thebibliography}{1}
\expandafter\ifx\csname natexlab\endcsname\relax\def\natexlab#1{#1}\fi
\expandafter\ifx\csname url\endcsname\relax
  \def\url#1{\texttt{#1}}\fi
\expandafter\ifx\csname urlprefix\endcsname\relax\def\urlprefix{URL }\fi

\bibitem[{Liao et~al.(2014)Liao, Huang, and Ting}]{liao2014}
Liao, J.-J., Huang, K.-N., Ting, P.-S., 2014. Optimal strategy of deteriorating
  items with capacity constraints under two-levels of trade credit policy.
  Applied Mathematics and Computation 233, 647--658.

\end{thebibliography}
}

\end{document}